\documentclass[11pt,reqno]{amsart}
\usepackage{amsmath} 
\usepackage{amssymb}
\usepackage{dsfont}
\usepackage[dvips,draft,final]{graphics}
\usepackage{amssymb,amsmath}
\usepackage[T1]{fontenc}
\usepackage{fancyhdr}

\usepackage{color}
\usepackage{graphicx}

\def\squarebox#1{\hbox to #1{\hfill\vbox to #1{\vfill}}}

\theoremstyle{plain}

\newtheorem{Thm}{Theorem}

\newtheorem{Def}{Definition}
   
\newtheorem{lem}{Lemma}

\newcommand{\im}{\mathfrak I}
\newcommand{\Div}{\textrm{div}}
\newcommand{\R}{\mathbb{R}}

\newcommand{\U}{{\mathcal U}}
\newcommand{\CI}{{\mathcal C}^{\infty}_{0}({\R}^{n}) }
\newcommand{\Ci}{{\mathcal C}^{\infty}_{0}(\Omega(t)) }
\newcommand{\C}{{\mathbb C}}

\newcommand{\half}{\frac{1}{2}}

\newcommand{\CII}{{\mathcal C}^{\infty}_{0}(\vert x\vert\leq b) }
\def\epsilon{\varepsilon}
\def\phi {\varphi}

\newtheorem{rem}{Remark}
\newtheorem{prop}{Proposition}
\newtheorem{defi}{Definition}

\providecommand{\abs}[1]{\left\lvert#1\right\rvert}
\providecommand{\norm}[1]{\left\lVert#1\right\rVert}
\numberwithin{equation}{section}
\renewcommand{\d}{\textrm{d}}
\renewcommand{\leq}{\leqslant}
\renewcommand{\geq}{\geqslant}
\providecommand{\abs}[1]{\left\lvert#1\right\rvert}
\providecommand{\norm}[1]{\left\lVert#1\right\rVert}

\begin{document}

\title[Local energy decay ]
{ Local energy decay  for the wave equation  with a time-periodic non-trapping metric  and moving obstacle}

\author[Y. Kian]{Yavar Kian}

\address {Centre de Physique Th\'eorique CNRS-Luminy, Case 907, 13288 Marseille, France}

\email{Yavar.Kian@cpt.univ-mrs.fr}
\maketitle
\begin{abstract}
Consider the mixed problem with Dirichelet condition associated to the wave equation $\partial_t^2u-\Div_{x}(a(t,x)\nabla_{x}u)=0$, where the scalar metric $a(t,x)$ is $T$-periodic in $t$ and uniformly equal to $1$ outside a compact set in $x$, on a $T$-periodic domain. Let $\mathcal U(t, 0)$ be the associated propagator. Assuming that the perturbations are non-trapping, we prove the meromorphic continuation of the cut-off resolvent of the Floquet operator $\mathcal U(T, 0)$ and  we establish sufficient conditions for local energy decay.
 
\end{abstract}

\section*{Introduction}
\renewcommand{\theequation}{\arabic{section}.\arabic{equation}}
\setcounter{equation}{0}
Let $\Omega$ be an open domain in ${\R}^{1+n}$, $n\geq3$ with $\mathcal C^\infty$ boundary $\partial\Omega$. Introduce the sets
\[\Omega(t)=\{x\in{\R}^n\ :\ (t,x)\in\Omega\},\quad O(t)={\R}^n\setminus \Omega(t), \quad t\in{\R}.\]
We assume that there exists $\rho_1>0$ such that for all $t\in{\R}$
\begin{equation}\label{obstacle1}O(t)\subset\{ x\ :\ \vert x\vert\leq\rho_1\}.\end{equation}
Moreover there exists $T>0$ such that
\begin{equation}\label{obstacle2}O(t+T)=O(t),\quad t\in{\R}.\end{equation}
For each $(t,x)\in\partial\Omega$, let $\nu(t,x)=(\nu_t(t,x),\nu_x(t,x))$ be the exterior unit normal vector  to $\partial\Omega$ at $(t,x)\in\partial\Omega$ pointing into $\Omega$. Then, we assume that there exists $0<c<1$ such that 
\begin{equation}\label{normal} \vert\nu_t\vert<c\vert\nu_x\vert.\end{equation}
Consider the following mixed problem
\begin{equation} \label{leprobleme}  \left\{\begin{aligned}
u_{tt}-\Div_{x}(a(t,x)\nabla_{x}u)&=0,\ \ (t,x)\in\Omega,\\
u_{|\partial\Omega}&=0\\
(u,u_{t})(s,x)=(f_{1}(x),f_{2}(x))&=f(x),\ \ x\in\Omega(s),\end{aligned}\right.\end{equation}
where the perturbation $a(t,x)\in \mathcal C^\infty({\R}^{n+1})$ is a scalar function which satisfies the conditions:
\begin{equation}\label{eq=perturbationA}\begin{array}{l}(i) \ C\geq a(t,x)\geq c>0,\   (t,x)\in{\R}^{n+1},\\
(ii)\ \textrm{ there exists }\rho>\rho_1\textrm{ such that }a(t,x)=1\textrm{ for }\vert x\vert\geq\rho,\\
(iii)\textrm{ there exists }T>0\textrm{ such that }a(t+T,x)=a(t,x),\   (t,x)\in{\R}^{n+1}.\\
\end{array}\end{equation}
Throughout this paper we assume   $n\geq3$. Consider the set $H(t)$ which is the closure of the space $\Ci\times\Ci$ with respect to the norm
\[\displaystyle\norm{f}_{H(t)}=\left(\underset{{\Omega(t)}}\int\left(\abs{\nabla_xf_1}^2+\abs{f_2}^2\right)\d x\right)^\half,\quad f=(f_1,f_2)\in\Ci\times\Ci.\]

Let us introduce  some general properties of solutions of \eqref{leprobleme}. We show, in Section 1, that for $f\in H(s)$ there exists a unique solution of \eqref{leprobleme} and we  introduce  the propagator
\begin{equation}\label{pro}\mathcal{U}(t,s):H(s)\ni(f_1,f_2)=f\mapsto \mathcal{U}(t,s)f=(u,u_t)(t,x)\in H(t)\end{equation}
with $u$ the solution of \eqref{leprobleme}. Moreover, we prove that $\mathcal{U}(t,s)$ is a bounded operator satisfying the following estimate
\begin{equation}\label{energy}\norm{\mathcal{U}(t,s)}_{\mathcal L(H(s),H(t))}\leq Be^{A\abs{t-s}}.\end{equation}

The goal of  this paper is to establish  sufficient conditions    for a local energy decay taking the form
\begin{equation} \label{eq=locA}\Vert\chi \mathcal U(t,s)\chi\Vert_{\mathcal L(H(s),H(t))}\leq C_{\chi}p(t-s),\quad t\geq s,\end{equation}
with $p(t)\in L^1({\R}^+)$ and $\chi\in\mathcal C^\infty_0(\vert x\vert\leq\rho+1)$. 

We study problem \eqref{leprobleme} under a {\bf non-trapping} condition. More precisely, let $\mathcal U(t,s,x,x_0)$ be the kernel of the propagator $\mathcal{U}(t,s)$ and consider the following
\begin{enumerate}
\item[$\rm(H1)$]
 For all $r>0$, there exists $T_1(r)>0$ such that 
 \[\mathcal U(t,s,x,x_0)\in\mathcal C^\infty\left(\{(t,s,x,x_0)\ :\ \abs{x}\leq r,\ \abs{x_0}\leq r,\ \abs{t-s}\geq T_1(r)\}\right).\]
\end{enumerate}
From   \cite{MS}, we know that singularities propagate along null-bicharacteristics (with consideration of their reflections from $\partial\Omega$). Thus, one can show that condition $\rm(H1)$ is equivalent to the requirement that all null-bicharacteristics of \eqref{leprobleme} with consideration of reflections
from $\partial\Omega$ go out to infinity as $\abs{t-s}\to+\infty$.
 Let us recall that the non-trapping condition (H1) is necessary for  \eqref{eq=locA}  since for some trapping perturbations we may have solutions with exponentially increasing  energy (see \cite{CR} for $\Omega={\R}^{1+n}$ and \cite{RP} for $a(t,x)=1$). On the other hand, even for non-trapping periodic perturbations some parametric resonances could lead to solutions with exponentially growing  energy (see \cite{CPR} for   time-periodic potentials). To exclude the existence of such solutions we must consider a second assumption.

Many authors  have investigated the  local  energy decay  of wave equations. The main hypothesis is that the perturbations are non-trapping. For $a(t,x)=a_0(x)$ independent of time and fixed obstacles,  the meromorphic continuation and estimates of the cut-off resolvent $ \chi\left(-\Div_x(a_0(x)\nabla_x.)-\lambda^2\right)^{-1}\chi$, where $\chi\in\CI$ and $\lambda\in\C$,  are the  main arguments for estimate \eqref{eq=locA} (see \cite{V3}, \cite{V1}, \cite{V5} and \cite{V4}). From  these results, by considering the connection between the Fourier transform in time of the solutions  and the  stationary operator $-\Div_x(a_0(x)\nabla_x.)-\lambda^2$, one can deduce \eqref{eq=locA}. For time dependent metric $a(t,x)$ or moving obstacle, since the domain or the Hamiltonian $-\Div_x(a(t,x)\nabla_x.)$ are time-dependent, we cannot apply these arguments. However, the analysis of the Floquet operator $\mathcal U(T,0)$ makes it possible to obtain \eqref{eq=locA} with  $T$-periodic perturbations and moving obstacle. In \cite{CS} the authors have extended the Lax-Phillips theory to  problem \eqref{leprobleme} with $a(t,x)=1$ and they have established a local energy decay \eqref{eq=locA}. By using the compactness of the local evolution operator, deduced from a propagation of singularities, and the RAGE theorem of Georgiev and Petkov (see \cite{GP}), Bachelot and Petkov  have shown in \cite{BP1} that in the case of odd dimensions, the decay of the local energy associated to the wave equation with time periodic potential  is exponential for initial data  with compact support  included in a subspace of finite codimension. Petkov has extended this result to even dimensions (see \cite{P2}), by using the meromorphic continuation of the cut-off resolvent of the Floquet operator associated to this problem.

Let us introduce the cut-off resolvent, associated to the Floquet operator $\mathcal{U}(T,0)$, defined by \[R_{\psi_1,\psi_2}(\theta)=\psi_1(\mathcal{U}(T,0)-e^{-i\theta})^{-1}\psi_2:\ H(0)\to H(0),\quad\psi_1,\psi_2\in{\CI}.\]
According to \eqref{energy}, $R_{\psi_1,\psi_2}(\theta)$ is a family of bounded operators analytic with respect to $\theta$ on $\{\theta\in\mathbb C\ :\ \im(\theta)>AT\}$. 
Applying some arguments of \cite{V2}, in Section 2, we show the meromorphic continuation of $R_{\psi_1,\psi_2}(\theta)$  to $\C$ for $n$ odd and to $\{\theta\in\mathbb{C}\  :\   \theta\notin 2\pi\mathbb{Z}+i\R^-\}$ for $n$ even. Let us recall that the meromorphic continuation of $R_{\psi_1,\psi_2}(\theta)$ is closely related to the asymptotic expansion of $\chi\mathcal{U}(t,0)\chi$, $\chi\in\CI$, as $t\to+\infty$ (see Section 2 and the main theorem in \cite{V2}).   Consequently, it seems natural to consider the meromorphic continuations of $R_{\psi_1,\psi_2}(\theta)$ that imply  \eqref{eq=locA}. Consider the following assumption.

\begin{enumerate}
\item[$\rm(H2)$] There exist $\phi_1,\phi_2\in{\CI}$, satisfying $\phi_1(x)=\phi_2(x)=1$ for $\abs{x}\leq\rho+T+2$, such that the operator $R_{\phi_1,\phi_2}(\theta)$ admits an analytic continuation from $\{\theta\in\mathbb{C}\  :\   \textrm{Im}(\theta) \geq A > 0\}$ to
$\{\theta\in\mathbb{C}\  :\   \textrm{Im}(\theta) \geq  0\}$, for $n \geq 3$, odd, and to $\{\theta\in\mathbb{C}\  :\   \textrm{Im}(\theta) >  0\}$ for $n \geq4$, even.
Moreover, for $n$ even, $R_{\phi_1,\phi_2}(\theta)$ admits a continuous continuation from $\{\theta\in\mathbb{C}\  :\   \textrm{Im}(\theta) > 0\}$ to $\{\theta\in\mathbb{C}\  :\   \textrm{Im}(\theta) \geq  0,\theta\neq 2k\pi, k\in\mathbb{Z}\}$ and we have
\[\limsup_{\substack{\lambda\to0 \\ \textrm{Im}(\lambda)>0}}\Vert R_{\phi_1,\phi_2}(\lambda)\Vert<\infty.\]
\end{enumerate}
Assuming (H1) and (H2) fulfilled, we obtain the following.
\begin{Thm} Assume \eqref{obstacle1}, \eqref{obstacle2}, \eqref{normal}, \eqref{eq=perturbationA}, $\rm(H1)$ and $\rm(H2)$ fulfilled. Then, estimate \eqref{eq=locA}
is fulfilled with
\begin{equation}\label{lalalu}\left\{\begin{aligned}p(t)&=e^{-\delta t}\textrm{ for $n\geq3$ odd},\\  p(t)&= \frac{1}{(t+1)\ln^2(t+e)}\textrm{ for $n\geq4$ even}.\end{aligned}\right.\end{equation}
\end{Thm}
Let us remark that, assuming (H1) fulfilled, $\rm(H2)$ is a   necessary and sufficient condition for estimate \eqref{eq=locA} with $p(t)$ satisfying  \eqref{lalalu}. Moreover, if $\rm(H2)$ is not fulfilled, even  the uniform estimate in time of the local energy $\norm{\chi\mathcal{U}(t,0)\chi}_{\mathcal L(H(0),H(t))}$ may not hold. For example, if $R_{\phi_1,\phi_2}(\theta)$ has a pole $\theta_0\in\mathbb C$ with $\im(\theta_0)>0$, one can establish  the estimate
\[\norm{\chi\mathcal U(t,0)\chi}_{\mathcal L(H(0),H(t))}\geq Ce^{\frac{\im(\theta_0)}{T}t}\]
and deduce  existence of a solution with compactly supported initial data and exponentially growing local energy.
 It has been established in \cite{CPR} that these phenomenon can occur even with a non-trapping condition. The goal of $\rm(H2)$ is to avoid existence of such solutions.

\begin{rem}
Let the metric $(a_{ij}(t,x))_{1\leq i,j\leq n}$ be such that for all $i,j=1\cdots n$ we have
\[\begin{array}{l}
\displaystyle(i)\ \textrm{ there exists }\rho>0\textrm{ such that }a_{ij}(t,x)=\delta_{ij},\textrm{ for }\vert x\vert\geq\rho,\textrm{ with $\delta_{ij}=0$ for $i\neq j$ and $\delta_{ii}=1$},\\
\displaystyle(ii)\textrm{ there exists }T>0\textrm{ such that }a_{ij}(t+T,x)=a_{ij}(t,x),\  \forall (t,x)\in{\R}^{n+1},\\
\displaystyle(iii)a_{ij}(t,x)=a_{ji}(t,x),\forall (t,x)\in{\R}^{n+1},\\
\displaystyle(iv)\textrm{ there exist } C>c>0 \textrm{ such that }C\vert\xi\vert^2\geq\sum_{i,j=1}^na_{ij}(t,x)\xi_i\xi_j\geq c\vert\xi\vert^2 ,\ \ \forall(t,x)\in{\R}^{1+n},\ \xi\in{\R}^n.\\
\end{array}\]

If we replace $a(t,x)$ in \eqref{leprobleme} we get the following mixed problem
\begin{equation} \label{re}  \left\{\begin{aligned}
\displaystyle u_{tt}-\sum_{i,j=1}^n\frac{\partial}{\partial x_i}\left(a_{ij}(t,x)\frac{\partial}{\partial x_j}u\right)&=0,\ \ (t,x)\in\Omega,\\
u_{|\partial\Omega}&=0,\\
(u,u_{t})(s,x)=(f_{1}(x),f_{2}(x))&=f(x),\ \ x\in\Omega(s).\end{aligned}\right.\end{equation}
All the results of this paper  remain valid for the mixed problem \eqref{re} and their proofs follow from the same arguments.
\end{rem}

Notice that the estimate
\begin{equation} \label{a}\Vert\psi_1 \mathcal U(NT,0)\psi_2\Vert_{\mathcal L({H(0)})}\leq \frac{C_{\psi_1,\psi_2}}{(N+1)\ln^2(N+e)},\quad N\in\mathbb N,\end{equation}
implies (\ref{eq=locA}). On the other hand, if (\ref{a}) is valid, the assumption (H2) for $n$ even is fulfilled. Indeed, for large $A>>1$ and $ \textrm{Im}(\theta)\geq AT$ we have
\[R_{\psi_1,\psi_2}(\theta)=-e^{i\theta}\sum_{N=0}^{\infty}\psi_1\mathcal U(NT,0)\psi_2e^{iN\theta}\]
and applying (\ref{a}), we conclude that $R_{\psi_1,\psi_2}(\theta)$ admits an analytic continuation from\\
 $\{\theta\in\mathbb{C}\  :\   \textrm{Im}(\theta) \geq A > 0\}$   to $\{\theta\in\mathbb{C}\  :\   \textrm{Im}(\theta) >  0\}$. Moreover, $R_{\psi_1,\psi_2}(\theta)$ is bounded for $\theta\in{\R}$. In Section 4, we give some examples of metrics $a(t,x)$ and moving obstacle $O(t)$ such that  (\ref{a}) is fulfilled.  

\vspace{0,5cm}
\ \\

\section{General properties}
The purpose of this section is to establish some general properties of solutions of problem \eqref{leprobleme}. We will study the global well posedness of \eqref{leprobleme} and  we will prove estimate \eqref{energy}.  We start by fixing the notion of solutions of \eqref{leprobleme}.
\begin{Def}
A distribution $u(t,x)\in D'(\Omega)$ is called a solution of \eqref{leprobleme} if the following conditions hold:
\begin{enumerate}
\item[$\rm(i)$] $\left(u(t,.),u_t(t,.)\right)\in H(t)$ for each $t\in{\R}$; extended inside $O(t)$ by setting $u(t,x)=0$, the functions
\[t\longmapsto \nabla_xu(t,.),\quad t\longmapsto u_t(t,.)\]
are continuous with values in $L^2({\R}^n)$,
\item[$\rm(ii)$] $\left(u(s,.),u_t(s,.)\right)=(f_1,f_2)=f$
\item[$\rm(iii)$] $\partial_t^2u-\Div_x(a(t,x)\nabla_xu)=0$ in $\Omega$ in the sense of distributions.
\end{enumerate}\end{Def}
In the next result we obtain the existence and uniqueness of solutions of \eqref{leprobleme}.
\begin{Thm} Assume \eqref{obstacle1}, \eqref{obstacle2}, \eqref{normal} and \eqref{eq=perturbationA} fulfilled. Then, for each $f\in H(s)$ there exists a unique solution $u(t,.)$ of  \eqref{leprobleme} with the property that for each $t>0$ 
\begin{equation}\label{th2a} \sup_{\substack{\abs{t-s}\leq D \\ \abs{s}\leq2D}}\norm{\left(u(t,.),u_t(t,.)\right)}_{H(t)}\leq C_D\norm{f}_{H(s)}\end{equation}\end{Thm}
\begin{proof}
First we treat the existence and uniqueness of the solution for small $\abs{t-s}$. Given $z\in\Omega(s)$, consider the cone 
\[C_{z,s}=\{ (t,x)\in{\R}^{1+n}\ :\ \abs{x-z}\leq\abs{t-s}\}.\]
For $\abs{t-s}$ small enough and for $z$ outside a small neighborhood of $\partial\Omega(s)$ we obtain $C_{z,s}\subset \Omega$. Consequently, for $(t,x)\in C_{z,s}$ the solution $u(t,x)$ of the mixed problem coincides with the solution of the Cauchy problem
\begin{equation} \label{eq=lepbA}  \left\{\begin{aligned}
u_{tt}-\Div_{x}(a(t,x)\nabla_{x}u)&=0,\ \ (t,x)\in{\R}\times{\R}^{n},\\
(u,u_{t})(s,x)=(f_{1}(x),f_{2}(x))&=f(x),\ \ x\in{\R}^n,\end{aligned}\right.\end{equation}
with $f$ extended by $0$ for $x\in O(s)$. Thus, for $\abs{t-s}\leq \epsilon$ and $\epsilon$ sufficiently small, we will determine $u(t,x)$ in some small neighborhood of $\partial \Omega\cap\{\abs{t-s}\leq\epsilon\}$. Given $(s,z)$ with $z\in\partial\Omega(s)$, we establish the existence and uniqueness of $u(t,x)$ in some space-time neighborhood of $(s,z)$. Covering the compact set $\{s\}\times\partial\Omega(s)$ by a finite number of such neighborhoods and using the local uniqueness result for the points where these neighborhoods overlap, we deduce the existence and uniqueness for small $\abs{t-s}$. Introduce in a neighborhood of $(s,z)$, $z\in\partial\Omega(s)$, local coordinates $(t,y)$, $y'=(y_1,\ldots,y_{n-1})$, so that $(s,z)$ is transformed into $(0,0)$, while the boundary $\partial\Omega$ is given by $y_n=g(t,y')$ with $g$ a $\mathcal C^\infty$ function such that $\nabla_{y'}g(0,0)=0$.  Since 
\[\nu(t,y',g(t,y'))=\frac{1}{\sqrt{1+\abs{g_t(t,y')}^2+\abs{\nabla_{y'}g(t,y')}^2}}(-g_t(t,y'),-\nabla_{y'}g(t,y'),1),\]
statement \eqref{normal} implies that 
\[\abs{g_t(t,y')}<c\left(\abs{\nabla_{y'}g(t,y')}+1\right).\] 
Thus, we have 
$\abs{g_t(0,0)}<c$.
If we choose a sufficiently small neighborhood of $(0,0)$ we can assume that $\abs{g_t(t,y')}<c$.
Changing variables
\[x_j=y_j,\quad j=1,\ldots,n-1,\quad x_n=y_n-g(t,y')\]
we transform 
\[\partial_t^2-\Div_x(a(t,x)\nabla_x\cdot)\]
into the operator $P(t,x,\partial_t,\partial_x)$ with principal symbol 
\[\begin{aligned}\sigma(P(t,x,\partial_t,\partial_x))=&-\tau^2 +2g_t\tau\xi_n-2\xi_nb(t,x)\xi'\cdot\nabla_{x'}g+ b(t,x)\abs{\xi'}^2\\ &+\left(b(t,x)\abs{\nabla_{x'}g}^2-g_t^2+b(t,x)\right)\xi_n^2,\end{aligned}\]
where $b(t,x)=a(t,y)$.
Here $(\tau,\xi',\xi_n)$ are the variable dual to $(t,x',x_n)$. Statement \eqref{normal} and property \eqref{eq=perturbationA} imply that 
\begin{equation}\label{th2b}b(t,x)\abs{\nabla_{x'}g}^2-g_t^2+b(t,x)>0.\end{equation}
Consider the problem 
\begin{equation}\label{th2c}\left\{\begin{aligned} P(t,x,\partial_t,\partial_x)u&=0 \ \textrm{ in }\ \R_t\times\R^{n-1}_{x'}\times\R^+_{x_n},\\ 
u(t,x',0)&=0 \ \textrm{ in }\ \R_t\times\R^{n-1}_{x'},\\ (u(0,x),u_t(0,x))&=f(x).\end{aligned}\right.\end{equation}
We suitably extend the coefficients of $P(t,x,\partial_t,\partial_x)$ to ${\R}^{1+n}$ preserving the strict hyperbolicity of $P(t,x,\partial_t,\partial_x)$ with respect to $t$. For the mixed problem \eqref{th2c} we can apply the results of Miyatake \cite{Mi} and H\"ormander \cite{H}, Chapter XXIV. Notice that the inequality \eqref{th2b} guarantees that the boundary $x_n=0$ is timelike in the sense of H\"ormander \cite{H}. The result of Miyatake \cite{Mi} says that if 
$\nabla_xf_1,\ f_2\in L^2_{\textrm{loc}}\left(\R^{n-1}_{x'}\times\overline{\R^+_{x_n}}\right)$ with $ f_1=f_2$ for $ x_n=0$,
then for $\abs{t}\leq\delta$ there exists a unique solution $u(t,x)\in H^1_{\textrm{loc}}\left(\R^{n-1}_{x'}\times\overline{\R^+_{x_n}}\right)$ of \eqref{th2c} satisfying the estimate
\[\sum_{j+\abs{\beta}}\norm{\partial^j_t\partial_x^{\beta}u(t,x)}_{L^2_{\textrm{loc}}\left(\R^{n-1}_{x'}\times\overline{\R^+_{x_n}}\right)}\leq C_\delta\sum_{j+\abs{\beta}}\norm{\partial^j_t\partial_x^{\beta}u(0,x)}_{L^2_{\textrm{loc}}\left(\R^{n-1}_{x'}\times\overline{\R^+_{x_n}}\right)}\]
with a constant $C_\delta$ depending on $\delta$. Notice that \eqref{eq=perturbationA} implies that the boundary $x_n=0$ is non-characteristic for $P(t,x,\partial_t,\partial_x)$. So $u(t,x)\in\mathcal C^\infty\left(\overline{\R^+_{x_n}}; D'(\R^n)\right)$ (see Theorem B.2.9 in H\"ormander \cite{H}) and the trace $u_{|x_n=0}$ is meaningful. The same argument shows that $\nabla_xu(t,.)$ and $u_t(t,.)$ depend continuously on $t$. Thus we obtain the existence and uniqueness of the solution of \eqref{leprobleme} in $\Omega\cap\{\abs{t-s}\leq\epsilon\}$. We can determine $\epsilon>0$ uniformly with respect to $s$, provided $\abs{s}\leq2D$. Making a construction by steps of length $\epsilon$, we cover the interval $\abs{t-s}\leq D$ and the proof is complete.
\end{proof}

Following Theorem 2, we can introduce the propagator $\mathcal U(t,s)$ defined by \eqref{pro}. Combining the results of Theorem 2 and the  periodicity of $O(t)$ and $a(t,x)$, we deduce the following.

\begin{prop}Assume \eqref{obstacle1}, \eqref{obstacle2}, \eqref{normal} and \eqref{eq=perturbationA} fulfilled. Then, we have
\begin{equation}\label{p1a}\mathcal U(t+T,s+T)=\mathcal U(t,s),\end{equation}
\begin{equation}\label{p1b}\norm{\U(t,s)}_{\mathcal L(H(s),H(t))}\leq Be^{A\abs{t-s}}.\end{equation}\end{prop}
\begin{proof} The proof of \eqref{p1a} is trivial. Let us show estimate \eqref{p1b}. Applying \eqref{th2a}, we obtain
\[\sup_{\abs{s},\abs{t}\leq T}\norm{\U(t,s)}_{\mathcal L(H(s),H(t))}=C<\infty.\]
Let $t,s\in\R$ and let $0\leq t',s'<T$ be such that $t=lT+t'$ and $s=kT+s'$ with $k,l\in\mathbb Z$. Then, applying \eqref{p1a}, we obtain 
\[\U(t,s)=\U(t',0)\U((k-l)T,0)\U(s',0)=\U(t',0)\U(T,0)^{k-l}\U(s',0).\]
It follows that
\[\norm{\U(t,s)}_{\mathcal L(H(s),H(t))}\leq C^2(1+C)^{\abs{k-l}}\leq C^2e^{\ln(1+C)\abs{k-l}}\leq C^2e^{\ln(1+C)\abs{t-s}}\]
and we obtain \eqref{p1b} with $A=\ln(1+C)$.\end{proof}

Notice that, combing the arguments used in the proof of Theorem 2 with  estimate \eqref{p1b},  we can show  that the Duhamel's principal holds. 
Let $P_1$ and $P_2$ be the projectors of $\mathbb C^2$ defined by \[  P_1(h)=h_1,\quad P_2(h)=h_2,\quad h=(h_1,h_2)\in\mathbb C^2\]
and let $P^1,P^2\in\mathcal L({\C},{\C}^2)$ be defined by
\[  P^1(h)=(h,0),\quad P^2(h)=(0,h),\quad h\in\mathbb C.\]
Denote by $V(t,s): L^2(\Omega(s))\to \dot{H}^1(\Omega(t))$ the operator defined  by
\[V(t,s)=P_1\mathcal U(t,s)P^2.\]
Notice that for $h\in L^2(\Omega(s))$, $w=V(t,s)h$ is the solution of

\[  \left\{\begin{aligned}
\partial_t^2(w)-\Div_{x}(a(t,x)\nabla_{x}w)&=0,\\
w_{|\partial\Omega}&=0,\\
(w,\partial_tw)_{\vert t=s}&=(0,h).\end{aligned}\right.\]
Let $g(t,x)$ be a function defined on $\Omega$ such that, for $A_1>A$ (with $A$ the constant of \eqref{p1b}), $e^{-A_1t}g(t,x)\in L^2(\Omega)$ and $g(t,x)=0$ for $\abs{x}\geq b$ with $b\geq\rho+1$. Then there exists a unique solution $v$ of 
\[  \left\{\begin{aligned}
\partial_t^2(v)-\Div_{x}(a(t,x)\nabla_{x}v)&=g(t,x),\\
v_{|\partial\Omega}&=0,\\
(v,\partial_tv)_{\vert t=s}&=(0,0).\end{aligned}\right.\]
Moreover, this solution can be written in the following way
\begin{equation}\label{Du}v(t,.)=\int_s^tV(t,\tau)g(\tau,.)\d \tau.\end{equation}

\section{The meromorphic continuation of the cut-off resolvent $R_{\psi_1,\psi_2}(\theta)$}
The goal of this section is to prove the meromorphic continuation of $R_{\psi_1,\psi_2}(\theta)$, assuming $\rm(H1)$ fulfilled. The main result of this section is the following.
\begin{Thm} Assume $\rm(H1)$, \eqref{obstacle1}, \eqref{obstacle2}, \eqref{normal} and \eqref{eq=perturbationA} fulfilled. Let $\psi_1,\ \psi_2\in{\CI}$. Then,   $R_{\psi_1,\psi_2}(\theta)$ admits a meromorphic continuation from $\{\theta\in\C\ :\ \im(\theta)>AT\}$ to ${\C}$ for $n\geq3$ odd and to $\C'=\{\theta\in\C\ :\ \theta\notin2\pi\mathbb{Z}+i\R^-\}$ for $n\geq4$ even. Moreover, for $n\geq4$ even, there exists $\epsilon_0>0$ such that, for $\abs{\theta}\leq\epsilon_0$, we have
\begin{equation}\label{th4a}R_{\psi_1,\psi_2}(\theta)=\sum_{k\geq-m}\sum_{j\geq-m_k}R_{kj}\theta^k(\log\theta)^{-j}.\end{equation}
Here $R_{k,j}\in\mathcal L(H(0))$ and, for $k<0$ or $j>0$, $R_{k,j}$ is a finite rank operator.\end{Thm}
To prove Theorem 4, we will use some results of  \cite{V2} and \cite{Ki3}. For this purpose, we introduce some tools and  definitions of \cite{V2}.

Let $\gamma\in\mathcal C^\infty(\R)$ be such that $\gamma(t)=1$ for $t\geq -\frac{2T}{3}-\frac{T}{10}$ and $\gamma(t)=0$ for $t\leq -\frac{2T}{3}-\frac{2T}{10}$. Set
\[V_1(t,s)=\gamma(t-s)V(t,s).\]
We recall that the Fourier-Bloch-Gelfand transform $F$ is defined by 
\[F(\phi)(t,\theta)=\sum_{k=-\infty}^{+\infty}\left(\phi(t+kT,\cdot)e^{ik\theta}\right),\quad \phi\in\mathcal C^\infty_0(\R\times\R^n).\]
Applying \eqref{p1b}, for $\im(\theta)>AT$, with $A>0$ the constant of \eqref{p1b}, we can define 
\[F(\chi_1V_1(t,s)\chi_2)(t,\theta)=\sum_{k=-\infty}^{+\infty}\left(\chi_1V_1(t+kT,s)\chi_2e^{ik\theta}\right),\quad\chi_1,\ \chi_2\in{\CI}\]
and 
\[F'(\chi_1V_1(t,s)\chi_2)(t,\theta)=e^{i\frac{t\theta}{T}}F(\chi_1V_1(t,s)\chi_2)(t,\theta),\quad\chi_1,\ \chi_2\in{\CI}.\]
We will use the following definition of meromorphic continuation of a family of bounded operators.
\begin{defi}Let $H_1$ and $H_2$ be Hilbert spaces. A family of bounded operators
 $Q(t,s,\theta):H_1\rightarrow H_2$ is said to be meromorphic with respect to $\theta$ in a domain $D\subset\mathbb C$, if $Q(t,s,\theta)$ is meromorphically dependent on $\theta$ for $\theta\in D$ and for any pole  $\theta=\theta_0$  the coefficients of the negative powers of  $\theta-\theta_0$ in the appropriate Laurent extension are finite-rank operators. \end{defi}
 Denote $\mathbb C'=\{ z\in\mathbb C\ :\ z\neq 2k\pi-i\mu,\ k\in\mathbb{Z},\ \mu\geq0\}$ and consider the following meromorphic continuation.
\begin{defi} We say that the family of operators $Q(t,s,\theta)$, which are $\mathcal C^\infty$ with respect to $t$ and $s$, for $t\in{\R}$ and $0\leq s\leq\frac{2T}{3}$, and $T$-periodic with respect to $t$, has the property $(S')$ if:
1) for odd $n$ the operators $Q(t,s,\theta),\ \theta\in\mathbb{C},$ and its derivatives with respect to $t$ form a finitely-meromorphic family; 2) For even $n$ the operators $Q(t,s,\theta)$ and its derivatives with respect to $t$ form a finitely-meromorphic family for $\theta\in\mathbb{C}'$ . Moreover, in a neighborhood of $\theta=0$ in $\mathbb C'$, $Q(t,s,\theta)$ has the form
\begin{equation}\label{eq=VG}Q(t,s,\theta)=\theta^{-m}\sum_{j\geq0}\left(\frac{\theta}{R_{t,s}(\log \theta)}\right)^jP_{j,t,s}(\log \theta)+C(t,s,\theta),\end{equation}
where $C(t,s,\theta)$ is analytic with respect to $\theta$, $R_{t,s}$ is a polynomial, the $P_{j,t,s}$ are polynomials of order at most $l_j$ and $\log$ is the logarithm defined on $\mathbb{C}\setminus i{\R}^-$. Moreover, $C(t,s,\theta)$ and the coefficients of  the polynomials $R_{t,s}$ and $P_{j,t,s}$ are $\mathcal C^\infty$ and $T$-periodic with respect to $t$ and $\mathcal C^\infty$ with respect to $s$ for $0\leq s\leq\frac{2T}{3}$.\end{defi}
\begin{rem} Notice that if $Q(t,s,\theta)$ satisfies $(S')$ then $\partial_tQ(t,s,\theta)$ satisfies also $(S')$.\end{rem}

In \cite{V2} Vainberg proposed a general approach to problems with time-periodic perturbations including potentials, moving obstacles and high order operators, provided that the perturbations are non-trapping. One of the main results of \cite{V2} is the following.
\begin{Thm}\emph{(Theorem 10, \cite{V2})}
Assume that the mixed problem \eqref{leprobleme} is well posed, the Duhamel's principal holds and let \eqref{p1b} and $\rm(H1)$ be fulfilled. Then, for all  $b\geq\rho+1$, there exists $T_2(b)>T_1(b)$ and an operator 
\[R(t,s):\ L^2(\Omega(s))\to \dot{H}^1(\Omega(t))\]
such that the following conditions are fulfilled:
\begin{enumerate}
\item[$\rm(i)$] $R(t+T,s+T)=R(t,s)$,
\item[$\rm(ii)$] $R(t,s)$  is bounded,
\item[$\rm(iii)$] for all $\chi_1,\ \chi_2\in\CII$, $F'(\chi_1R(t,s)\chi_2)(t,\theta)$ admits a meromorphic continuation to the lower half plane satisfying property $(S')$ and  $\chi_1R(t,s)\chi_2=\chi_1V(t,s)\chi_2$ for $t-s\geq T_2(b)$.
\end{enumerate} 
\end{Thm}
In \cite{V2} Vainberg established the result of Theorem 4 for $s=0$. In \cite{Ki3} it has been proven that this result can be generalized to $0\leq s\leq\frac{2T}{3}$. Combining these results with the properties established in Section 1, we obtain a  meromorphic continuation of the Fourier-Bloch-Gelfand transform of   the solutions of \eqref{leprobleme} with initial data $(0,g)$ and their derivatives of order $1$ with respect to $t$.
\begin{lem} Assume $\rm(H1)$, \eqref{obstacle1}, \eqref{obstacle2}, \eqref{normal} and \eqref{eq=perturbationA} fulfilled. Then, for all $\psi_1,\ \psi_2\in{\CI}$ and all $0\leq s\leq\frac{2T}{3}$,
\[F'(\psi_1V_1(t,s)\psi_2)(T,\theta)\quad\textrm{and}\quad F'(\psi_1\partial_tV_1(t,s)\psi_2)(T,\theta)\]
admit a meromorphic continuation with respect to $\theta$, continuous with respect to $s\in\left[0,\frac{2T}{3}\right]$,  from $\{\theta\in\C\ :\ \im(\theta)>AT\}$ to ${\C}$ for $n\geq3$ odd and to $\C'=\{\theta\in\C\ :\ \theta\notin2\pi\mathbb{Z}+i\R^-\}$ for $n\geq4$ even. Moreover, for $n\geq4$ even, there exists $\epsilon_0>0$ such that, for $\abs{\theta}\leq\epsilon_0$, we have
\begin{equation}\label{l1a}F'(\psi_1V_1(t,s)\psi_2)(T,\theta)=\sum_{k\geq-m}\sum_{j\geq-m_k}Q_{kj}(s)\theta^k(\log\theta)^{-j}.\end{equation}
\begin{equation}\label{l1b}F'(\psi_1\partial_tV_1(t,s)\psi_2)(T,\theta)=\sum_{k\geq-m}\sum_{j\geq-m_k}S_{kj}(s)\theta^k(\log\theta)^{-j}.\end{equation}
Here $Q_{kj}(s),\ S_{kj}(s)\in\mathcal L(H(s),H(0))$ and are continuous with respect to $s$ for $0\leq s\leq\frac{2T}{3}$.\end{lem}
\begin{proof}
According to  Section 1, the mixed problem \eqref{leprobleme} is well posed, the Duhamel's principal holds, and  \eqref{p1b}, \eqref{Du} are fulfilled. Thus, we can apply the results of Theorem 4.
Choose $b\geq\rho+1$  such that supp$\psi_1\cup$supp$\psi_2\subset\{x\ :\ \abs{x}\leq b\}$. Take $h_b\in\mathcal C^\infty(\R)$ such that $h_b(t)=1$ for $t\geq T_2(b)+\frac{6T}{5}$ and $h_b(t)=0$ for $t\leq T_2(b)+T$. Then, for all $0\leq s\leq\frac{2T}{3}$, statement (iii) of Theorem 4 implies
\[h_b(t)\psi_1V_1(t,s)\psi_2=h_b(t)\psi_1R(t,s)\psi_2.\]
Thus, $F'(h_b(t)\psi_1V_1(t,s)\psi_2)(t,\theta)$ admits a meromorphic continuation satisfying property $(S')$. From now on, we assume that $T_2(b)=k_0T$ with $k_0\in\mathbb N$. For $\im(\theta)>AT$ we have 
\begin{equation}\label{l1c}F'(\psi_1V_1(t,s)\psi_2)(T,\theta)=F'(h_b(t)\psi_1V_1(t,s)\psi_2)(T,\theta)+F'[(1-h_b(t))\psi_1V_1(t,s)\psi_2](T,\theta).\end{equation}
Since $1-h_b(t)=0$ for $t\geq T_2(b)+\frac{6T}{5}=(k_0+1)T+\frac{T}{5}$, for $\im(\theta)>AT$, we get
\[F'[(1-h_b(t))\psi_1V_1(t,s)\psi_2](T,\theta)=e^{i\theta}\left[\sum_{k=1}^{k_0+1}(\psi_1V(kT,s)\psi_2e^{ik\theta})+\gamma(-s)\psi_1V(0,s)\psi_2\right].\]
Thus, $F'[(1-h_b(t))\psi_1V_1(t,s)\psi_2](T,\theta)$ admits an analytic continuation to $\C$. Combining the meromorphic continuation of $F'(h_b(t)\psi_1V_1(t,s)\psi_2)(T,\theta)$, the analytic continuation of\\
 $F'[(1-h_b(t))\psi_1V_1(t,s)\psi_2](T,\theta)$ and representation \eqref{l1c}, we obtain  the  meromorphic continuation of  $F'(\psi_1V_1(t,s)\psi_2)(T,\theta)$. It remains to prove the meromorphic continuation of
  $F'(\psi_1\partial_tV_1(t,s)\psi_2)(T,\theta)$. Notice that 
\[\partial_tV(t,s)=P_2\U(t,s)P^2\]
and, for $\im(\theta)>AT$, $F'(\psi_1V_1(t,s))(t,\theta)$ is well defined.
For $\im(\theta)>AT$, we have
\[\begin{aligned}\partial_t\left[F'(h_b(t)\psi_1V_1(t,s)\psi_2)(t,\theta)\right]=&\frac{i\theta}{T}F'(h_b(t)\psi_1V_1(t,s)\psi_2)(t,\theta)+F'(h_b'(t)\psi_1V_1(t,s)\psi_2)(t,\theta)\\
&+F'(h_b(t)\psi_1\partial_tV_1(t,s))(t,\theta)\end{aligned}\]
Thus, for $\im(\theta)>AT$, we get
\begin{equation}\label{l1d}\begin{aligned}F'(h_b(t)\psi_1\partial_tV_1(t,s)\psi_2)(t,\theta)=&\partial_t\left[F'(h_b(t)\psi_1V_1(t,s)\psi_2)(t,\theta)\right]-F'(h_b'(t)\psi_1V_1(t,s)\psi_2)(t,\theta)\\
&-\frac{i\theta}{T}F'(h_b(t)\psi_1V_1(t,s)\psi_2)(t,\theta)\end{aligned}\end{equation}
Since $F'(h_b(t)\psi_1V_1(t,s)\psi_2)(t,\theta)$ admits a meromorphic continuation satisfying property $(S')$ 
\[\partial_t\left[F'(h_b(t)\psi_1V_1(t,s)\psi_2)(t,\theta)\right]\quad\textrm{and}\quad \frac{i\theta}{T}F'(h_b(t)\psi_1V_1(t,s)\psi_2)(t,\theta)\]
admit also a meromorphic continuation satisfying property $(S')$. Moreover, since $h_b'(t)=0$ for $t\geq T_2(b)+\frac{6T}{5}$, $F'(h_b'(t)\psi_1V_1(t,s)\psi_2)(t,\theta)$ admits an analytic continuation with respect to $\theta$. It follows that $F'(h_b(t)\psi_1\partial_tV_1(t,s)\psi_2)(t,\theta)$ admits a meromorphic continuation satisfying property $(S')$. We conclude by repeating the arguments used for proving the meromorphic continuation of $F'(\psi_1V_1(t,s)\psi_2)(T,\theta)$.\end{proof}
Consider the operator defined by
 \[U(t,s)=P_1\mathcal U(t,s)P^1.\] For all $h\in\dot{H}^1(\Omega(s))$, $w=U(t,s)h$ is the solution of
\[  \left\{\begin{aligned}
\partial_t^2w-\Div_{x}(a(t,x)\nabla_{x}w)&=0,\\
w_{|\partial\Omega}&=0,\\
(w,w_t)_{\vert t=s}&=(h,0).\end{aligned}\right.\]
Let $\gamma_1\in\mathcal C^\infty(\R)$ be such that $\gamma_1(t)=1$ for $t\geq-\frac{T}{20}$ and $\gamma_1(t)=0$ for $t\leq-\frac{T}{10}$. Set 
\[U_1(t,s)=\gamma_1(t-s)U(t,s).\]
Applying \eqref{p1b}, for $\im(\theta)>AT$ and $\psi_1\ \psi_2\in\CI$ we can define $F'(\psi_1U_1(t,s)\psi_2)(t,\theta)$ and $F'(\psi_1\partial_tU_1(t,s)\psi_2)(t,\theta)$. From the results of Lemma 1 we obtain the following meromorphic continuation of $F'(\psi_1U_1(t,s)\psi_2)(T,\theta)$ and $F'(\psi_1\partial_tU_1(t,s)\psi_2)(T,\theta)$.
\begin{lem} Assume $\rm(H1)$, \eqref{obstacle1}, \eqref{obstacle2}, \eqref{normal} and \eqref{eq=perturbationA} fulfilled. Then, for all $\psi_1,\ \psi_2\in{\CI}$,
\[F'(\psi_1U_1(t,s)\psi_2)(T,\theta)\quad\textrm{and}\quad F'(\psi_1\partial_tU_1(t,s)\psi_2)(T,\theta)\]
admit a meromorphic continuation with respect to $\theta$, continuous with respect to $s\in\left[0,\frac{2T}{3}\right]$,  from $\{\theta\in\C\ :\ \im(\theta)>AT\}$ to ${\C}$ for $n\geq3$ odd and to $\C'$ for $n\geq4$ even. Moreover, for $n\geq4$ even, there exists $\epsilon_0>0$ such that, for $\abs{\theta}\leq\epsilon_0$, we have
\begin{equation}\label{l2a}F'(\psi_1U_1(t,0)\psi_2)(T,\theta)=\sum_{k\geq-m}\sum_{j\geq-m_k}M_{kj}\theta^k(\log\theta)^{-j}.\end{equation}
\begin{equation}\label{l2b}F'(\psi_1\partial_tU_1(t,0)\psi_2)(T,\theta)=\sum_{k\geq-m}\sum_{j\geq-m_k}N_{kj}\theta^k(\log\theta)^{-j}.\end{equation}
Here $M_{kj},\ N_{kj}\in\mathcal L(H(0))$ and, for $k<0$ or $j>0$, $M_{kj},\ N_{kj}$ are a finite rank operator.\end{lem}
\begin{proof}
Let $\alpha\in\mathcal C^\infty({\R})$ be such that $\alpha(t)=0$ for $t\leq\frac{T}{2}$ and $\alpha(t)=1$ for $t\geq \frac{2T}{3}$. For all $h\in\dot{H}^1({\R}^n)$, $Z=\alpha(t)U(t,0)h$ is the solution of
\begin{equation}\label{eq=thm2R} \left\{\begin{aligned}
\partial_t^2Z-\Div_{x}(a(t,x)\nabla_{x}Z)&=[\partial_t^2,\alpha](t)U(t,0)h,\\
Z_{|\partial\Omega}&=0,\\
(Z,\partial_tZ)_{\vert t=0}&=(0,0).\end{aligned}\right.\end{equation}
We deduce from the Cauchy problem (\ref{eq=thm2R}) the following representation
\begin{equation}\label{eq=thm2S}U(t,0)=\alpha(t)U(t,0)=\int_0^tV(t,s)[\partial_t^2,\alpha](s)U(s,0)\d s,\quad t\geq T.\end{equation}
Since $[\partial_t^2,\alpha](t)=0$ for $t>\frac{2T}{3}$, the formula (\ref{eq=thm2S}) becomes
\[U(t,0)=\int_0^{\frac{2T}{3}}V(t,s)[\partial_t^2,\alpha](s)U(s,0)\d s,\quad  t\geq T.\]
Let $R>0$ be such that 
\[\textrm{supp}\psi_1\cup\textrm{supp}\psi_2\subset\{x\ :\ \abs{x}\leq R\}.\]
Choose $b=R+\rho+T+1$ and take $\chi\in\CII$ such that
\[\chi(x)=1\quad \textrm{for}\quad \abs{x}\leq R+\rho+T.\]
The finite speed of propagation implies 

\begin{equation}\label{l2c}\psi_1U(t,0)\psi_2=\int_0^{\frac{2T}{3}}\psi_1V(t,s)\chi[\partial_t^2,\alpha](s)U(s,0)\psi_2\d s,\quad  t\geq T.\end{equation}
Thus, for $\im(\theta)>AT$, we obtain
\[\begin{aligned}F'(\psi_1U_1(t,0)\psi_2)(T,\theta)=&F'\left[\int_0^{\frac{2T}{3}}\psi_1V_1(t,s)\chi[\partial_t^2,\alpha](s)U(s,0)\psi_2\d s\right](T,\theta)\\
&-\int_0^{\frac{2T}{3}}\psi_1V_1(0,s)\chi[\partial_t^2,\alpha](s)U(s,0)\psi_2\d s+e^{i\theta}\psi_1\psi_2\end{aligned}\]
and it follows
\[\begin{aligned}F'\left[\psi_1U_1(t,0)\psi_2\right](T,\theta)=&\int_0^{\frac{2T}{3}}F'\left[\psi_1V_1(t,s)\chi\right](T,\theta)[\partial_t^2,\alpha](s)U(s,0)\psi_2\d s\\
&-\int_0^{\frac{2T}{3}}\psi_1V_1(0,s)\chi[\partial_t^2,\alpha](s)U(s,0)\psi_2\d s+e^{i\theta}\psi_1\psi_2.\end{aligned}\]
Combining this representation with the meromorphic continuation of $F'(\psi_1V_1(t,s)\chi)(T,\theta)$ established in Lemma 1, we prove the meromorphic continuation of 
$F'(\psi_1U(t,0)\psi_2)(T,\theta)$ as well as \eqref{l2a}. It remains to prove the meromorphic continuation of $F'(\psi_1\partial_tU(t,0)\psi_2)(T,\theta)$.
Let $\beta\in{\CI}$. The formula (\ref{l2c}) implies that, for $t\geq T$, we have
\[\partial _tU(t,0)\beta=\int_0^{\frac{2T}{3}}\partial_tV(t,s)[\partial_t^2,\alpha](s)U(s,0)\beta \d s.\]
By density, this leads to
\[\psi_1\partial _tU(t,0)\psi_2=\int_0^{\frac{2T}{3}}\psi_1\partial _tV(t,s)\chi[\partial_t^2,\alpha](s)U(s,0)\psi_2\d s,\quad  t\geq T\]
and, for $\im(\theta)>AT$, we get
\[\begin{aligned}F'(\psi_1\partial_tU_1(t,0)\psi_2)(T,\theta)=&\int_0^{\frac{2T}{3}}F'(\psi_1\partial_tV_1(t,s)\chi)(T,\theta)[\partial_t^2,\alpha](s)U(s,0)\psi_2\d s\\
&-\int_0^{\frac{2T}{3}}\psi_1V_1(0,s)\chi[\partial_t^2,\alpha](s)U(s,0)\psi_2\d s.\end{aligned}\]
We conclude  by combining this representation with the results of Lemma 1.\end{proof}
\textit{Proof of Theorem 4.} By definition, we can write 
\[\gamma_1(t)\psi_1\U(t,0)\psi_2=\left(\begin{array}{cc}\gamma_1(t)\psi_1U(t,0)\psi_2&\gamma_1(t)\psi_1V(t,0)\psi_2\\ \gamma_1(t)\psi_1\partial_tU(t,0)\psi_2&\gamma_1(t)\psi_1\partial_tV(t,0)\psi_2\end{array}\right).\]
Moreover, for $\im(\theta)>AT$, we have
\begin{equation}\label{th4c}F'\left[\gamma_1(t)\psi_1\U(t,0)\psi_2\right](T,\theta)=e^{i\theta}\sum_{k=0}^{\infty}\left(\psi_1\U(T+kT,0)\psi_2e^{ik\theta}\right)=-e^{-i\theta}R_{\psi_1,\psi_2}(\theta)\end{equation}
and we obtain
\[\begin{aligned}R_{\psi_1,\psi_2}(\theta)&=-e^{i\theta}F'\left[\gamma_1(t)\psi_1\U(t,0)\psi_2\right](T,\theta)\\ &=-\left(\begin{array}{cc}e^{i\theta}F'(\psi_1U_1(t,0)\psi_2)(T,\theta)&e^{i\theta}F'(\psi_1V_1(t,0)\psi_2)(T,\theta)\\ e^{i\theta}F'(\psi_1\partial_tU_1(t,0)\psi_2)(T,\theta)&e^{i\theta}F'(\psi_1\partial_tV_1(t,0)\psi_2)(T,\theta)\end{array}\right).\end{aligned} \]
Thus, combining the results of Lemma 1 and Lemma 2, we prove Theorem 4.\hspace{4cm}$\square$

\section{Local energy decay}
The goal of this section is to prove Theorem 1, assuming (H1) and (H2) fulfilled. For this purpose, we show how assumption (H2) alter the meromorphic continuation of $R_{\psi_1,\psi_2}(\theta)$ established in Section 2. Then, by integrating on a suitable contour, we prove the local energy decay. We  treat separately the case of odd and  even dimensions. We start with $n$ odd.
\begin{lem}Assume $n\geq3$ odd, \eqref{obstacle1}, \eqref{obstacle2}, \eqref{normal}, \eqref{eq=perturbationA}, $\rm(H1)$ and $\rm(H2)$  fulfilled. Then, for all $\psi_1,\ \psi_2\in\mathcal C^\infty_0(\abs{x}\leq\rho+1)$, we get
\begin{equation}\label{l3a}\norm{\psi_1\U(t,s)\psi_2}_{\mathcal L(H(s),H(t))}\leq Ce^{-\delta(t-s)},\quad t\geq s.\end{equation}\end{lem}

\begin{proof}  
Notice that, for $\im(\theta)>AT$, $F'\left[\gamma_1(t)\phi_1\U(t,0)\phi_2\right](t,\theta)$ is $T$-periodic with respect to $t$ and $2\pi$-periodic with respect to $\theta$(see \cite{V2} Theorem ). Applying \eqref{th4c},  we get
\begin{equation}\label{l3b}F'\left[\gamma_1(t)\phi_1\U(t,0)\phi_2\right](dT,\theta)=F'\left[\gamma_1(t)\phi_1\U(t,0)\phi_2\right](T,\theta)=-e^{-i\theta}R_{\phi_1,\phi_2}(\theta).\end{equation}
Moreover, from \cite{V2} we have the following inversion formula (see Lemma 1 of \cite{V2})
\begin{equation}\label{l3c}\phi_1\U(dT,0)\phi_2=\frac{1}{2\pi}\int_{[i(A+1)T-\pi,i(A+1) T+\pi]}\hspace{-2cm}e^{-id\theta}F'\left[\gamma_1(t)\phi_1\U(t,0)\phi_2\right](T,\theta)\d \theta.\end{equation}
We will show \eqref{l3a}, by combining these  statements  with assumption (H2). 

First, assumption (H2) and \eqref{l3b} imply that $F'\left[\gamma_1(t)\phi_1\U(t,0)\phi_2\right](dT,\theta)$ has no poles on $\{\theta\ :\ \im(\theta\geq0)\}$. It follows that there exists $\delta>0$ such that $F'\left[\gamma_1(t)\phi_1\U(t,0)\phi_2\right](dT,\theta)$ has no poles on $\{\theta\ :\ \im(\theta)\geq-\delta T,\ -\pi\leq\textrm{Re}(\theta)\leq\pi\}$. Consider the contour $C_1$ defined by 
\[\mathcal C_1=[i(A+1)T+\pi,i(A+1)T-\pi]\cup[i(A+1)T-\pi,-i\delta T-\pi]\cup[-i\delta T-\pi,-i\delta T+\pi] \cup[-i\delta T+\pi,i(A+1)T+\pi].\]
The Cauchy formula implies
\[\int_{C_1}e^{-id\theta}F'\left[\gamma_1(t)\phi_1\U(t,0)\phi_2\right](T,\theta)\d \theta=0.\] 
Also, since  $F'\left[\gamma_1(t)\phi_1\U(t,0)\phi_2\right](T,\theta)$ is $2\pi$-periodic with respect to $\theta$ we have
\[\int_{[i(A+1)T-\pi,-i\delta T-\pi]}\hspace{-2cm}e^{-id\theta}F'\left[\gamma_1(t)\phi_1\U(t,0)\phi_2\right](T,\theta)\d \theta=-\int_{[-i\delta T+\pi,i(A+1)T+\pi]}\hspace{-2cm}e^{-id\theta}F'\left[\gamma_1(t)\phi_1\U(t,0)\phi_2\right](T,\theta)\d \theta\]
and we obtain
\begin{equation}\label{l3d}\int_{[i(A+1)T-\pi,i(A+1) T+\pi]}\hspace{-3cm}e^{-id\theta}F'\left[\gamma_1(t)\phi_1\U(t,0)\phi_2\right](T,\theta)\d \theta=\int_{[-i\delta T-\pi,-i\delta T+\pi]}\hspace{-2cm}e^{-id\theta}F'\left[\gamma_1(t)\phi_1\U(t,0)\phi_2\right](T,\theta)\d \theta.\end{equation}
It is obvious that
\[\int_{[-i\delta T-\pi,-i\delta T+\pi]}\hspace{-2cm}e^{-id\theta}F'\left[\gamma_1(t)\phi_1\U(t,0)\phi_2\right](T,\theta)\d \theta=e^{-\delta(dT)}\int_{[-\pi,\pi]}e^{-id\theta}F'\left[\gamma_1(t)\phi_1\U(t,0)\phi_2\right](T,\theta-i\delta T)\d \theta\]
and combining this with \eqref{l3d} and the inversion formula \eqref{l3c}, we get
\begin{equation}\label{l3e}\norm{\phi_1\U(dT,0)\phi_2}_{\mathcal L(H(0))}\leq Ce^{-\delta(dT)}.\end{equation}
Now let $\psi_1,\ \psi_2\in\mathcal C^\infty_0(\abs{x}\leq\rho+1)$, and let $t,\ s\in\R$ be such that $t\geq s$. we write $t=t'+mT$ and $s=s'+kT$ with $0\leq t',s'<T$ and $m,\ k\in\mathbb{N}$. The finite speed of propagation implies
\[\psi_1\U(t,s)\psi_2=\psi_1\U(t',0)\phi_1\U((m-k)T,0)\phi_2\U(0,s')\psi_2.\]
Then, applying \eqref{l3e} and Theorem 2, we obtain
\[\norm{\psi_1\U(t,s)\psi_2}_{\mathcal L(H(s),H(t))}\leq C\norm{\phi_1\U((m-k)T,0)\phi_2}_{\mathcal L(H(0))}\leq C'e^{-\delta((m-k)T)}\leq C'e^{-\delta(t-s)}.\]
\end{proof}

\begin{lem}Assume $n\geq4$ even, \eqref{obstacle1}, \eqref{obstacle2}, \eqref{normal}, \eqref{eq=perturbationA}, $\rm(H1)$ and $\rm(H2)$  fulfilled. Then, for all $\psi_1,\ \psi_2\in\mathcal C^\infty_0(\abs{x}\leq\rho+1)$, we get
\begin{equation}\label{l4a}\norm{\psi_1\U(t,s)\psi_2}_{\mathcal L(H(s),H(t))}\leq Cp(t-s),\quad t\geq s\end{equation}
with
\[p(t)=\frac{1}{(t+1)\ln^2(t+e)}.\]\end{lem}

\begin{proof}  
Repeating the arguments used in the proof of Lemma 3, we obtain that $F'\left[\gamma_1(t)\phi_1\U(t,0)\phi_2\right](T,\theta)$ has no poles on $\{ \theta\in\mathbb C'\ :\ \textrm{Im}(\theta)\geq0\}$. Moreover, representation \eqref{th4a} implies that there exists $\epsilon_0>0$ such that for $\theta\in\mathbb C'$ with $\vert\theta\vert\leq\epsilon_0$ we have
\begin{equation}\label{eq=thm2F}F'\left[\gamma_1(t)\phi_1\U(t,0)\phi_2\right](T,\theta)=\sum_{k\geq-m}\sum_{j\geq-m_k}R_{kj}\theta^k(\log\theta)^{-j}\end{equation}
and assumption (H2) implies that in this representation we have $R_{kj}=0$ for $k<0$ or $k=0$ and $j<0$. It follows that, for $\theta\in\mathbb C'$ with $\vert\theta\vert\leq\epsilon_0$, we obtain the following representation
\begin{equation}\label{eq=thm2G}F'\left[\gamma_1(t)\phi_1\U(t,0)\phi_2\right](T,\theta)=A(\theta)+B\theta^{m_0}\log(\theta)^{-\mu}+ \underset{\theta\rightarrow0}o\left(\theta^{m_0}\log(\theta)^{-\mu}\right)\end{equation}
with $A(\theta)$ analytic with respect to $\theta$ for $\vert\theta\vert\leq\epsilon_0$ , $B$ a finite-dimensional operator, $m_0\geq0$  and $\mu\geq 1$. Since $F'\left[\gamma_1(t)\phi_1\U(t,0)\phi_2\right](T,\theta)$ has no poles on $\{ \theta\in\mathbb C'\ :\ \textrm{Im}(\theta)\geq0\}$, there exists $\displaystyle0<\delta\leq\frac{\epsilon_0}{T}$ and $0<\nu<\epsilon_0$  sufficiently small such that  $F'\left[\gamma_1(t)\phi_1\U(t,0)\phi_2\right](T,\theta)$ has no poles on  \[\{ \theta\in\mathbb C\ :\ \textrm{Im}(\theta)\geq-\delta T ,\ -\pi\leq  \textrm{Re}(\theta)\leq-\nu,\ \nu\leq  \textrm{Re}(\theta)\leq\pi\}.\] Consider the contour $\Sigma=\Gamma_1\cup\omega\cup\Gamma_2$ where $\Gamma_1=[-i\delta T-\pi,-i\delta T-\nu]$, $\Gamma_2=[-i\delta+\nu,-i\delta+\pi]$. The contour  $\omega$  of $\mathbb C$, is a curve connecting $-i\delta T-\nu$ and $-i\delta T+\nu$ symmetric with respect to the axis $ \textrm{Re}(\theta)=0$. The part of $\omega$ lying in $\{\theta\ :\  \textrm{Im}(\theta)\geq0\}$ is a half-circle
with radius $\nu$, $\omega\cap\{\theta\ :\   \textrm{Re}(\theta)<0,\  \textrm{Im}(\theta)\leq0\}=[-\nu-i\delta T,-\nu]$ and $\omega\cap\{\theta\ :\   \textrm{Re}(\theta)>0,\  \textrm{Im}(\theta)\leq0\}=[\nu,\nu-i\delta T]$. Thus, $\omega$ is included in the
region where we have no poles of $F'\left[\gamma_1(t)\phi_1\U(t,0)\phi_2\right](T,\theta)$. Consider the closed contour 
\[\mathcal C_2=[i(A+1)T+\pi,i(A+1)T-\pi]\cup[i(A+1)T-\pi,-i\delta T-\pi]\cup\Sigma\cup[-i\delta T+\pi,i(A+1)T+\pi].\]
An application of the Cauchy formula yields
\[\int_{\mathcal C_2}e^{-id\theta}F'\left[\gamma_1(t)\phi_1\U(t,0)\phi_2\right](T,\theta)\d\theta=0.\]
Applying the same arguments as those used in the proof of Lemma 3, we obtain
\[\displaystyle\int_{[i(A+1)T-\pi,i(A+1) T+\pi]}\hspace{-2cm}F\left[\gamma_1(t)\phi_1\U(t,0)\phi_2\right](dT,\theta)\d \theta=\int_\Sigma e^{-id\theta}F'\left[\gamma_1(t)\phi_1\U(t,0)\phi_2\right](T,\theta)\d \theta\]
and the inversion formula \eqref{l3c} implies 

\begin{equation}\label{l4b}\phi_1\U(dT,0)\phi_2=\frac{1}{2\pi}\int_\Sigma e^{-id\theta}F'\left[\gamma_1(t)\phi_1\U(t,0)\phi_2\right](T,\theta)\d \theta,\quad d\in\mathbb N.\end{equation}

Combining this representation with \eqref{eq=thm2G} and applying some arguments used in Lemma 2 and Lemma 3 of \cite{Ki3}, we obtain \eqref{l4a}.

\end{proof}

Combining the results of Lemma 3 and Lemma 4, we prove Theorem 1.

\section{ Examples of metrics $a(t,x)$ and obstacles $O(t)$} 

In this section we will apply some properties of solutions of the wave equations with   non-trapping metrics  independent of $t$ and fixed obstacle  to construct time periodic metrics and moving obstacles such that conditions (H1) and  (H2) are fulfilled. For this purpose, we assume that (H1) is fulfilled for the metrics $a(t,x)$ and obstacle $O(t)$ that we consider and we will establish examples for (H2).   In order to prove (H2), we will modify  the size $T$ of the period of $a(t,x)$. This choice is justified by the properties of $\mathcal U(t,s)$.

Let $T_1>0$ and let $((a_T(t,x), O_T(t)))_{T\geq T_1}$ be a family of couples of functions and obstacles such that the following conditions are fulfilled:
\begin{enumerate}
\item[$\rm(H3i)$] $a_T(t,x)$ and $O_T(t)$ are $T$-periodic with respect to  $t$ and $a_T(t,x)$ satisfies (\ref{eq=perturbationA}),
\item[$\rm(H3ii)$] for all $T\geq T_1$, if $(a(t,x), O(t))=(a_T(t,x), O_T(t))$ then conditions \eqref{obstacle1}, \eqref{obstacle2}, \eqref{normal} and $\rm(H1)$ are fulfilled,
 \item[$\rm(H3iii)$] there exist a function $a_1(x)$ and an obstacle $O$ independent of $t$ such that for  \[(a(t,x), O(t))=(a_1(x), O)\] condition $\rm(H1)$ is fulfilled and, for all $T_1\leq t\leq T$, we have $a_T(t,x)=a_1(x)$ and $O_T(t)=O$.
 \end{enumerate}

Let $H$ be the closure of the space $\mathcal C^\infty_0(\R^n\setminus O)\times\mathcal C^\infty_0(\R^n\setminus O)$ with respect to the norm
\[\displaystyle\norm{f}_{H}=\left(\underset{{\R^n\setminus O}}\int\left(\abs{\nabla_xf_1}^2+\abs{f_2}^2\right)\d x\right)^\half,\quad f=(f_1,f_2)\in\mathcal C^\infty_0(\R^n\setminus O)\times\mathcal C^\infty_0(\R^n\setminus O).\]
Consider the following Cauchy problem
\begin{equation} \label{eq=exC}  \left\{\begin{array}{c}
v_{tt}-\Div_{x}(a_1(x)\nabla_{x}v)=0,\quad t\in\R,\ x\in\R^n\setminus O\\
v_{|\partial O}=0,\quad t\in\R,\\
(v,v_{t})(0)=f,\end{array}\right.\end{equation}
and the associate propagator
\[\mathcal{V}(t):H\ni f\longmapsto (v,v_t)(t)\in H.\]
Let $u$ be solution of (\ref{eq=lepbA}). For $T_1\leq t\leq T$ we have
\[\partial_t^2 u -\Div_x(a_1(x)\nabla_x u)=\partial_t^2 u -\Div_x(a_T(t,x)\nabla_x u)=0.\]
It follows that for $(a(t,x),O(t))=(a_T(t,x), O_T(t))$ we get
\begin{equation} \label{eq=exD}\mathcal{U}(t,s)=\mathcal{V}(t-s),\quad T_1\leq s<t\leq T\end{equation}
and 
\begin{equation} \label{tru}H(t)=H,\quad T_1\leq t\leq T.\end{equation}
The asymptotic expansion of $\chi \mathcal{V}(t)\chi$ as $t\rightarrow+\infty$   has been studied by many authors (see  \cite{V1}, \cite{V3} and \cite{V4}). It has been proven that, for non-trapping metrics and for $n\geq3$, the local energy decreases. To prove (H2), we will apply the following result. 
\begin{Thm}\label{t13}
Assume $n\geq3$. Let $\phi\in{\CI}$. Then, we have
\begin{equation} \label{eq=thm11A}\Vert \phi\mathcal{V}(t)\phi\Vert_{\mathcal{L}(H)}\leq C_{\phi}p(t)\end{equation}
with 
\[\left\{\begin{aligned}p(t)=& e^{-\delta t}\ \ \textrm{for $n$ odd,}\\ p(t)=&\left\langle t\right\rangle^{1-n}\ \ \textrm{for $n$ even}.\end{aligned}\right.\]\end{Thm}
Estimate (\ref{eq=thm11A}) has been established by Vainberg in  \cite{V3}, \cite{V1} but also by Vodev in \cite{V5} and \cite{V4}.
  For $n\geq4$ even we will use the following identity.

\begin{lem}\label{l4}
Let $\psi\in\mathcal{C}^\infty_0(\vert x\vert\leq\rho+1+T_1)$ be such that $\psi=1$, for $\vert x\vert\leq\rho+\half+T_1$. Then, we have
\begin{equation} \label{eq=lem5A}\mathcal{U}(T_1,0)-\mathcal{V}(T_1)=\psi(\mathcal{U}(T_1,0)-\mathcal{V}(T_1))=(\mathcal{U}(T_1,0)-\mathcal{V}(T_1))\psi.\end{equation}
\end{lem}
\begin{proof}
First, notice that \eqref{tru} implies $H(0)=H$.
Now, choose $g\in H(0)=H$ and let $w$ be the function defined by $(w,w_t)(t)=\mathcal{U}(t,0)(1-\psi)g$. The finite speed of propagation implies that, for  $0\leq t\leq T_1$ and $\vert x\vert\leq\rho+\half$, we get $w(t,x)=0$. Moreover, we have
\begin{equation} \label{eq=lem5B} \Div_x(a_1(x)\nabla_x)=\Delta_x=\Div_x(a(t,x)\nabla_x),\quad\textrm{for } \vert x\vert>\rho.\end{equation}
Thus, $w$ is solution on $0\leq t\leq T_1$ of the problem
\[  \left\{\begin{array}{c}
w_{tt}-\Div_{x}(a_1(x)\nabla_{x}w)=0,\quad t\in\R,\ x\in\R^n\setminus O\\
w_{|\partial O}=0,\quad t\in\R,\\

(w,w_{t})(0)=(1-\psi)g\end{array}\right.\]
and it follows that
\begin{equation} \label{eq=lem5C}(\mathcal{U}(T_1,0)-\mathcal{V}(T_1))(1-\psi)=0.\end{equation}
Now, let  $u$ and $v$ be the functions defined by $(u,u_t)(t)=\mathcal{U}(t,0)g$ and $(v,v_t)(t)=\mathcal{V}(t)g$ with $g\in H$. Applying (\ref{eq=lem5B}), we can easily show that on  $(1-\psi)u$ is the solution of
\[  \left\{\begin{array}{c}
\partial_t^2((1-\psi)u))-\Delta_x((1-\psi)u))=[\Delta_x,\psi]u,\\
(((1-\psi)u),((1-\psi)u)_{t})(0)=(1-\psi)g,\end{array}\right.\]
and $(1-\psi)v$ is the solution of
\[  \left\{\begin{array}{c}
\partial_t^2(((1-\psi)v))-\Delta_x((1-\psi)v))=[\Delta_x,\psi]v,\\
(((1-\psi)v),((1-\psi)v)_{t})(0)=(1-\psi)g.\end{array}\right.\]
We have
\begin{equation} \label{eq=lem5D}(1-\psi)(\mathcal{U}(T_1,0)-\mathcal{V}(T_1))=0.\end{equation}
Combining (\ref{eq=lem5C}) and (\ref{eq=lem5D}), we get (\ref{eq=lem5A}).\end{proof}

 Combining the  arguments  used in the proofs of Lemma 7, 8 and 9 and Theorem 14 of \cite{Ki3} with the identity \eqref{eq=lem5A}, we obtain the following.

\begin{Thm}\label{t14}
Assume $n\geq3$   and let $((a_T(t,x),O_T(t)))_{T\geq T_1}$ satisfy $\rm(H3i)$, $\rm(H3ii)$, $\rm(H3iii)$. Then, for $T$ large enough and for $(a(t,x),O(t))=(a_T(t,x),O_T(t))$, assumption $\rm(H2)$ is fulfilled.
\end{Thm}

\end{document}